\documentclass[11pt]{amsart}

\usepackage{amsmath,amssymb,amsbsy,amsthm,latexsym,
         amsopn,amstext,amsxtra,euscript,amscd,amsthm, mathtools, dsfont}
         
\usepackage{fullpage,amssymb,amsmath}
\usepackage{enumerate}
\usepackage[shortlabels]{enumitem} 
\usepackage{comment}
\usepackage{hyperref} 
\usepackage[capitalise,nameinlink]{cleveref}
\usepackage{url}
\usepackage{xcolor}
\usepackage{amsmath,amsthm,amssymb, graphicx, multicol, array, mathtools, enumerate,enumitem,hyperref,algorithm,algorithmic,tikz}
\usepackage{fancyhdr}
\usepackage{upgreek}
\usepackage{physics}
\usepackage{xcolor}
\usepackage{pythontex} 

\usepackage{bbm}
\usepackage{mathrsfs} 
\newtheorem*{corollary*}{Corollary}
\newtheorem*{theorem*}{Theorem}
\newtheorem{theorem}{Theorem} 

\newtheorem{proposition}[theorem]{Proposition}
\newtheorem{lemma}[theorem]{Lemma}

\newtheorem{question}[theorem]{Question}
\newtheorem*{question*}{Question}

\theoremstyle{definition}

\newtheorem*{remark}{Remark}

\theoremstyle{remark}

\numberwithin{equation}{section}

\newcommand{\N}{\mathbb{N}}
\newcommand{\Z}{\mathbb{Z}}

\newcommand{\R}{\mathbb{R}}

\newcommand{\veps}{\varepsilon}

\renewcommand{\epsilon}{\varepsilon}

\begin{document}

\title{A counterexample to symmetry of $L^p$ norms of eigenfunctions}

\author{Gabriel Beiner}
\address{Department of Mathematics\\ University of Toronto\\ 40 St. George Street\\ Toronto\\
ON\\ Canada\\ M5S 2E4
} 
\email{gabriel.beiner@mail.utoronto.ca}

\author{Nancy Mae Eagles}
\address{ Trinity College, University of Cambridge, Cambridge, CB21TQ, United Kingdom
} 
\email{nme22@cam.ac.uk}

\author{William Verreault}
\address{
D\'{e}partement de Math\'{e}matiques et de Statistique\\ 
Universit\'{e} Laval\\ 
Qu\'ebec\\
QC\\
G1V 0A6 \\
Canada} 
\email{william.verreault.2@ulaval.ca}

\author{Runyue Wang}
\address{Department of Mathematics \& Statistics\\ McMaster University\\ Hamilton\\
ON\\ Canada\\ L8S 4L8}
\email{wangr109@mcmaster.ca}

\maketitle

\vspace*{-4mm}

\begin{abstract}
We answer a question of Jakobson and Nadirashvili on the asymptotic behavior of the $L^p$ norms of positive and negative parts of eigenfunctions of the Laplacian. More precisely, we show that there exists a sequence of eigenfunctions $\psi_n$ on the flat $d$-torus for $d\geq 3$, with eigenvalues $\lambda_n\to\infty$ as $n\to\infty$, such that the ratio 
$
\|\psi_n\chi_{\{\psi_n>0\}}\|_p / \|\psi_n\chi_{\{\psi_n<0\}}\|_p 
$
does not tend to $1$ as $n\to\infty$ for $1<p\leq \infty$.
Our argument is elementary and computer-assisted.

\end{abstract}

\section{Introduction} \label{intro}

In \cite{MR1900757}, Jakobson and Nadirashvili contribute to the extensive literature on the behavior of $L^p$ norms of eigenfunctions of the Laplacian by investigating the relationship between positive and negative parts of Laplace eigenfunctions. In particular, they study quasi-symmetry properties of their $L^p$ norms. The authors conclude their paper by asking the following.
\begin{question} \label{q:J-N}
For $\psi_\lambda$ a nonconstant eigenfunction of the Laplacian with eigenvalue $\lambda$, does $$
\frac{\|\psi_\lambda\chi_{\{\psi_\lambda>0\}}\|_p}{\|\psi_\lambda\chi_{\{\psi_\lambda<0\}}\|_p}\to 1
$$ as $\lambda\to\infty$ for $1<p<\infty$ on a given manifold?
\end{question}
The same question is raised in \cite{jakobson2001geometric}. The aim of this paper is to answer \cref{q:J-N} in the negative.\\

The investigation of the symmetry between positive and negative parts of eigenfunctions of the Laplacian is motivated by predictions of the random wave conjectures in quantum chaos and is intimately related to Yau's conjecture (see \cite{jakobson2001geometric} and \cite{logunov2019review}). Donnelly and Fefferman \cite{MR943927} proved Yau's conjecture for real analytic manifolds, and as a corollary obtained the following important result regarding these symmetry distribution problems.
\begin{theorem}[\cite{MR943927}, Corollary 7.10] \label{thm:D-F}
For $M$ a compact analytic Riemannian manifold and $\psi_\lambda$ a Laplace eigenfunction, there exists a constant $C>0$, depending only on the manifold, such that
$$
\frac{1}{C} \leq \frac{\emph{\text{vol}}\{x\in M: \psi_\lambda(x)>0\}}{\emph{\text{vol}}\{x\in M: \psi_\lambda(x)<0\}} \leq C.
$$
\end{theorem}
Nadirashvili gave a different proof for analytic surfaces in \cite{MR1112199}.
It is worth noting that the analogous statement for smooth manifolds remains open and is sometimes referred to as quasi-symmetry conjecture \cite{logunov2019review}. 
\\

It is natural to inquire about the behavior of the ratio \[\text{vol}\{\psi_\lambda>0\} / \text{vol}\{\psi_\lambda<0\}\] in the semiclassical limit as $\lambda\to\infty$. In that regard, it has been conjectured that the constant $C$ in \cref{thm:D-F} can be taken to be $1$ asymptotically as $\lambda\to\infty$, which has sometimes been referred to as the symmetry conjecture \cite{LogunovVideo, eigentorus}. This is supported by Berry's conjecture for manifolds with constant negative curvature \cite{MR489542}, but is believed to hold much more generally \cite{LogunovVideo}.
Recently, Mart\'inez and Torres de Lizaur \cite{eigentorus} showed that this is false by giving a computer-assisted counterexample on the flat torus $\R^d/(2\pi\Z)^d$ for $d\geq 3$. This is in some way the simplest counterexample, since they also proved that the constant $C$ in \cref{thm:D-F} can be taken to be $1$ for the flat $2$-torus. 

 The extension of the previous questions to $L^p$ norms is natural. Nadirashvili \cite{MR1112199} showed the $L^\infty$ analogue of \cref{thm:D-F}, while Jakobson and Nadirashvili \cite{MR1900757} proved the following $L^p$ generalization of \cref{thm:D-F} for $p\geq 1$.

\begin{theorem}[Jakobson--Nadirashvilli]
Given a nonconstant real eigenfunction of the Laplacian $\psi_\lambda$ on a smooth compact manifold $M$ and $p\geq 1$, there exists $C>0$, depending only on $p$ and the manifold $M$, such that
\[\frac{1}{C} \leq \frac{\|\psi_\lambda \chi_{\{\psi_\lambda>0\}}\|_p}{\|\psi_\lambda \chi_{\{\psi_\lambda<0\}}\|_p} \leq C.\]
\end{theorem}
We remark that the results of Mart\'inez and Torres de Lizaur \cite{eigentorus} also prove that $C=1$ works for the flat $2$-torus in the previous theorem.
In the rest of this paper, we set out to provide a negative answer to \cref{q:J-N} on the standard flat torus $\mathbb{T}^d=\R^d/\Z^d$, $d\geq 3$, using the counterexample developed for the volume case in \cite{eigentorus}. We will present the argument for the $3$-torus, but it obviously extends to higher dimensions.

We adopt the following notation. We write
$$\psi_+ = \psi \chi_{\{\psi>0\}} \qquad \text{and} \qquad \psi_- = -\psi  \chi_{\{\psi<0\}},$$ where $\chi_A$ stands for the indicator function of the set $A$. Both of these functions are nonnegative by definition. We use $\psi_{\pm}$ when the statement applies to both $\psi_+$ and $\psi_-$. We will always integrate over the manifold $\mathbb{T}^3$, but we drop the index in the integrals for brevity. For the same reason, we omit the differential symbol $dV$. 

\begin{remark}
In \cref{q:J-N}, the endpoints $p=1$, $p=\infty$ are excluded. This is simply because for $p=1$, the answer is affirmative since the ratio is always equal to $1$ (indeed, $\psi=\psi_+-\psi_-$ and $\int\psi=0$ by Green's identities); while for $p=\infty$, it was shown in \cite{MR1900757} that the answer is negative using even zonal spherical harmonics. Although it was only stated that $\|\psi_+\|_{\infty}/\|\psi_-\|_{\infty}> 1$, this is sufficient to answer \cref{q:J-N} for $p=\infty$ on the sphere.
\end{remark}

 Finally, the main result of this paper, answering \cref{q:J-N}, is the following theorem. 
\begin{theorem} \label{thm:counter}
There exists a sequence of eigenfunctions $\psi_n$ on the flat $d$-torus for $d\geq 3$, with eigenvalues $\lambda_n\to\infty$ as $n\to\infty$, such that for all $1< p\leq \infty$,
\[\frac{\|\psi_{n,+}\|_p}{\|\psi_{n,-}\|_p} \not \to 1\]
as $n\to \infty$.
\end{theorem}

An upcoming paper \cite{BV} establishes the failure of asymptotic symmetry
of $L^p$ norms on the $2$-sphere for $p\geq 6$. Coupled with the present work, it provides counterexamples to \cref{q:J-N} for model spaces of zero and constant positive curvature. Whether this question has a positive answer for manifolds of constant negative curvature remains open.

\section{A counterexample to the $L^p$ symmetry conjecture} \label{sec:proof}
For the standard flat torus, and more generally for flat tori since they are quotients of $\R^d$ by lattices, it will be enough to prove that there exists some Laplace eigenfunction $\psi$ such that
$$
\|\psi_+\|_p\neq \|\psi_-\|_p
$$
for $p\in(1,\infty]$. Indeed, assuming this fact, define
$$\psi_n(x,y,z) = \psi(nx,ny,nz)$$ for $n\in \N$.
If $\psi$ is an eigenfunction with eigenvalue $\lambda$, it follows that
$\psi_n$ is an eigenfunction with eigenvalue $\lambda n^2$ for each $n$. We also see by a change of variables that
\[\int\psi_{n,\pm}^p(x,y,z)=\int\psi_{\pm}^p(nx,ny,nz)=\int\psi_{\pm}^p(x,y,z),\]
and so 
\[\frac{\|\psi_{n,+}\|_p}{\|\psi_{n,-}\|_p} =\frac{\|\psi_{+}\|_p}{\|\psi_{-}\|_p}
\neq 1,\]
even in the limit as $n\to\infty$. 

Thus,
\cref{thm:counter} will follow from the following proposition.
\begin{proposition} \label{prop:counter}
There exists an eigenfunction $\psi$ of the Laplacian on $\mathbb{T}^3$ such that
$$
\|\psi_+\|_p \neq \|\psi_-\|_p
$$
for $1<p\leq \infty$.
\end{proposition}
We shall prove this proposition with
$$\psi(x,y,z) =  \sin(2\pi(x+y))-\cos(2\pi(y-z))-\sin(2\pi(x+z)),$$
which is an eigenfunction on $\mathbb{T}^3 \cong \R^3/\Z^3$ with an eigenvalue of $8\pi^2$. 
This choice of eigenfunction is the same (up to scaling) as the counterexample in \cite{eigentorus}.
In particular, we shall prove that $\|\psi_{+}\|_p/\|\psi_{-}\|_p>1$. Since we know $\|\psi_+\|_1/\|\psi_-\|_1=1$, it will thus be enough to show that $\int \psi_+^p/\int\psi_-^p$ is strictly increasing in $p$ for $p\geq 1$. It is easy to compute
\begin{align} \label{eq:derivative}
\pdv{p}\Big(\frac{\int \psi_+^p}{\int\psi_-^p}\Big) = \frac{\int \psi_+^p}{\int\psi_-^p}\Big(\frac{\int \psi_+^p\log\psi_+}{\int\psi_+^p}- \frac{\int \psi_-^p\log\psi_-}{\int\psi_-^p}\Big),
\end{align}
and using the well-known inequalities
$1-1/x \leq \log x\leq x-1$ for $x >0$ in this expression, it follows that \eqref{eq:derivative} is at least
$$
\frac{\int \psi_+^p}{\int\psi_-^p}\Big(2-\frac{\int \psi_+^{p-1}}{\int\psi_+^p}-\frac{\int\psi_-^{p+1}}{\int\psi_-^p}\Big).
$$
For $p\geq 1$ and $q\geq 0$, let
\[f(p)=\frac{\int \psi_+^{p-1}}{\int\psi_+^p}\quad \text{and} \quad g(q)=\frac{\int \psi_-^{q+1}}{\int\psi_-^q}.\] 
We will then be done if we show
that $f(p)+g(p)<2$ for $p\geq 1$. Let us start with a proposition about the monotonicity behaviour of these functions.


\begin{lemma} \label{lem:mono}
The function $f(p)$ is decreasing for $p\geq 1$ while $g(q)$ is increasing for $q\geq 0$.
\end{lemma}
\begin{proof} 
Let $p>0$ and $\epsilon>0$. The Cauchy--Schwarz inequality yields
\begin{align*}
\int \psi_{\pm}^p &\leq \Big(\int\psi_{\pm}^{p-\epsilon}\int\psi_{\pm}^{p+\epsilon}\Big)^{1/2}.
\end{align*}
Since these integrals are all nonnegative, we can square both sides and rearrange to obtain
\begin{equation}\label{eq:epsilon}
\frac{\int \psi_{\pm}^p}{\int\psi_{\pm}^{p-\epsilon}}\leq \frac{\int \psi_{\pm}^{p+\epsilon}}{\int\psi_{\pm}^{p}}.
\end{equation}
Now fix $N\in\N$. The inequality \eqref{eq:epsilon} holds in particular for $\epsilon=N^{-1}$ and $p$ replaced with $p-k\epsilon$ whenever $p\geq 1$ and $0\leq k\leq N-1$ is an integer. Taking the product of these inequalities over $0\leq k\leq N-1$ yields telescoping products that give
\begin{equation}\label{eq:ineqinp}
\frac{\int \psi_{\pm}^p}{\int\psi_{\pm}^{p-1}}\leq \frac{\int \psi_{\pm}^{p+\epsilon}}{\int\psi_{\pm}^{p-1+\epsilon}}.
\end{equation}
Rearranging, we get in particular
\begin{equation}\label{eq:rational}
f(p)=\frac{\int \psi_{+}^{p-1}}{\int\psi_{+}^{p}}\geq \frac{\int \psi_{+}^{p-1+\epsilon}}{\int\psi_{+}^{p+\epsilon}} = f(p+\epsilon).
\end{equation}
Iterating this inequality implies that \eqref{eq:rational} holds when $\epsilon>0$ is a rational number, and by continuity of $f$ and the density of the rationals, it also holds for any $\epsilon>0$. It follows that $f$ is decreasing.

Finally, let $p=q+1$ in \eqref{eq:ineqinp} to obtain
\begin{align*}
g(q)=\frac{\int \psi_{-}^{q+1}}{\int\psi_{-}^{q}}\leq \frac{\int \psi_{-}^{q+1+\epsilon}}{\int\psi_{-}^{q+\epsilon}} = g(q+\epsilon)
\end{align*}
for $\epsilon=N^{-1}$ and $q\geq 0$. By the same reasoning as above, this inequality holds for any $\epsilon>0$ and hence $g$ is increasing.
\end{proof}

The proof of the next lemma is the key component of our counterexample. The computer-assisted argument it relies on is more intricate than the rest, so we dedicate the next section to it.

\begin{lemma} \label{lem:estimates}
The following estimates hold:
\begin{align*}
f(1) < 0.8, \qquad f(2) < 0.6, \qquad f(3) < 0.5,
\end{align*}
and 
\begin{align*}
\|g\|_\infty \leq 1.5, \qquad g(2) < 1.2, \qquad g(3) < 1.3.
\end{align*}
\end{lemma}

We now couple our observation on the monotonicity behaviour of $f$ and $g$ with these bounds 
by considering $p$ over three intervals. If $1\leq p\leq 2$, then by Lemmas \ref{lem:mono} and \ref{lem:estimates}, $f(p)+g(p)\leq f(1)+g(2) < 0.8+1.2=2$. Similarly, $f(p)+g(p)<1.9$ when $2\leq p\leq 3$ and $f(p)+g(p)<2$ when $3\leq p\leq \infty$. This concludes the proof of \cref{prop:counter} assuming Lemma \ref{lem:estimates}.

\section{Proof of \cref{lem:estimates}} \label{sec:key}
A simple optimization argument yields $\|\psi_-\|_{\infty}=3/2$. Thus, for all $p\geq 0$,
$$
|g(p)| = \frac{\int \psi_-^{p+1}}{\int \psi_-^{p}} \leq \|\psi_-\|_\infty=3/2.
$$

Verifying the other bounds requires more work. To do so, we  use a computer-assisted approach. Namely, we implement code in Python to obtain numerical bounds on $\int \psi_{\pm}^p$  (and hence on $g(p)$ and $f(p)$). We then estimate the cumulative error made in these computations. By proving that this error is sufficiently small, we guarantee that the rest of the bounds in \cref{lem:estimates} hold.

We adopt the following notation. We write $\mathbf{x}$ and $\mathbf{y}$ for elements of $[0,1]^3$. For a quantity $Q$, we write $\overline{Q}$ for its computer output and $E(Q)$ for the computational error made in approximating $Q$, which is such that $|Q-\overline{Q}|\leq E(Q)$. 

\subsection{Description of method to bound $\int\psi_{\pm}^p$}
We divide $[0,1]^3$ into $1500^3$ cubes, each of side length $\ell:=1/1500$. We then approximate separate midpoint Riemann sums using this partition, summing over all cubes where their midpoint value is either positive or negative, respectively. By purposefully over and underestimating these sums, we arrive at the numerical bounds on $\int\psi_{\pm}^p$ we desire.

Since we are using a computer-assisted approach, we need to know the sign of $\psi(\mathbf{x})$ in every cube with reasonable certainty, so that it is added to the correct Riemann sum approximating $\int\psi_{\pm}^p$. Without loss of generality, let us focus on the region $\psi(\mathbf{x})\leq 0$.
Since the norm of the gradient of $\psi$ is bounded above by $6\pi$, it follows by the mean value theorem and the Cauchy--Schwarz inequality that
$$|\psi(\mathbf{x})-\psi(\mathbf{y})| \leq  3\pi\sqrt{3}\ell$$
for any $\mathbf{y}$ in a cube of side length $\ell$ centered at $\mathbf{x}$. Denote this upper bound on the error in a given cube by $\alpha$.
Henceforth, we thus focus on the region $\psi(\mathbf{x})< -\alpha$, and to try to control rounding errors, we check the condition $\overline{\psi}(\mathbf{x})<-1.1\alpha$. If we are able to compute $\psi(\mathbf{x})$ within a certain accuracy, say if $E(\psi(\mathbf{x})) < 0.05\alpha$, then we will surely have that for any $\mathbf{y}$ in a cube centered at $\mathbf{x}$,
$$\psi(\mathbf{y})=\psi(\mathbf{y})-\psi(\mathbf{x})+\psi(\mathbf{x})-\overline{\psi}(\mathbf{x})+\overline{\psi}(\mathbf{x}) < \alpha+0.05\alpha-1.1\alpha<0.$$ 
We shall see later that our Python code indeed allows us to compute $\psi(\mathbf{x})$ with accuracy $<0.05\alpha$. 

We define five variables $S_{U,\pm}$, $S_{L,\pm}$, and $S_N$ via the following algorithm, where once again $\pm$ will be used when the statement applies to both the $+$ and $-$ variables. Initially, set these quantities equal to zero. Our goal is to have $S_{U,+}$ be an upper bound for the Riemann sums over cubes that admit only positive (respectively negative for $S_{U,-}$) values of $\psi$, and similarly for $S_{L,\pm}$ to be a lower bound, while $S_N$ will count the number of cubes left out of either of these sums. 
Thus, for the midpoint $\mathbf{x}$ of each cube in our partition of $[0,1]^3$, execute the following procedure:
 if $\psi(\mathbf{x})>1.1\alpha$, add $\ell^3(|\psi(\mathbf{x})|+1.1\alpha)^p$ to $S_{U,+}$ and $\ell^3(|\psi(\mathbf{x})|-1.1\alpha)^p$ to $S_{L,+}$; if $\psi(\mathbf{x})<-1.1\alpha$, add the same quantities to $S_{U,-}$ and $S_{L,-}$; and if $|\psi(\mathbf{x})|\leq 1.1\alpha$, add one to $S_N$.


 
 After performing this for all $1500^3$ cubes, we obtain bounds for the Riemann sums estimates of $\int\psi^p_\pm$. 
 For an upper bound on $\int\psi^p_\pm$, we take  $U_{\pm}:=S_{U,\pm}+ S_N\cdot(2.2\alpha
)^p\ell^3$, which comes from the fact that $|\psi(\mathbf{y})|<2.2\alpha$ if $\mathbf{y}$ is in a cube counted in $S_N$. For a lower bound on $\int\psi^p_\pm$, we simply take $L_{\pm}:=S_{L,\pm}$.\\

Implementing this in Python, the computer outputs of these bounds are shown in Table \ref{table:1} for $p=0,1,2,3,4$, where $p=0$ means we are just estimating the volume of positive cubes. This computation also gave $S_N = 17199000$. The commented Python code can be retrieved from \href{https://github.com/gbeiner/Counterexample-To-Symmetry-of-Lp-Norms-of-Eigenfunctions}{\tt{https://github.com/gbeiner/Counterexample-To-Symmetry-of-Lp-Norms-of-Eigenfunctions}}.

 \begin{table}[ht]
\begin{tabular}{|c | c | c | c |}
\hline
\\[-1em]
$p$ & $L_{+}$& $U_{+}$  \\ 
 \hline \hline

0 & 0.396101 &  0.401198  \\
\hline
1 &  0.511500  & 0.521105  \\
\hline
2 & 0.954454 &  0.979178  \\
\hline
3 & 2.063070 &  2.132510  \\
\hline
4 & 4.820963 &  5.021853  \\
\hline

\end{tabular}
\qquad \qquad 
\begin{tabular}{|c | c | c | c |}
\hline
\\[-1em]
$p$ & $L_{-}$& $U_{-}$  \\ 
 \hline \hline

0 & 0.598802 &  0.603899  \\
\hline
1 & 0.509072 &  0.523531  \\
\hline
2 & 0.520968 &  0.545691  \\
\hline
3 & 0.578638 &  0.616942  \\
\hline
4 & 0.676264 &  0.733500  \\
\hline

\end{tabular} 
\caption{Upper and lower Riemann sums estimates of $\int \psi_+^p$ (left) and $\int \psi_-^p$ (right) for $p=0,1,2,3,4$.}\label{table:1}
\end{table}
We then calculate that
\begin{equation} \label{eq:computations}
f(1) < 0.785,\quad f(2) < 0.546,\quad f(3) < 0.475,\quad g(2) < 1.185,\quad g(3) < 1.268.
\end{equation}

\subsection{Error control} \label{subsection:error}
We will be done if we can show that the absolute error made in computing each of the values in \eqref{eq:computations} is small enough for \cref{lem:estimates} to still hold when we add these errors to our approximations. In particular, comparing the statement of \cref{lem:estimates} with \eqref{eq:computations}, we see that if
$
E(U_{\pm}), E(L_{\pm}) < 0.01,
$
then our estimates hold. We also need to make sure that $E(|\psi(\mathbf{x})|)<0.05\alpha$.
Observe that
\begin{equation} \label{eq:errorRie}
E(L_{\pm})\leq E(U_{\pm}) \leq 1500^3\cdot E(\ell^3(|\psi(\mathbf{x})|+1.1\alpha)^p) + S_N\cdot E((2.2\alpha
)^p\ell^3).
\end{equation}

In the following, we write $\veps_m$ for machine epsilon in Python.
By the mean value theorem applied with respect to $\overline{2\pi(x+y)}$, the error in $\sin(2\pi(x+y))$ is at most $E(2\pi(x+y)) + \veps_m$, where the $\veps_m$ takes into consideration any rounding error. The same reasoning applies to the other two trigonometric functions making up $\psi$, and so
$$
E(|\psi(\mathbf{x})|) < 3(E(2\pi(x+y) + \veps_m) + 2\veps_m.
$$



To upper bound the previous errors, we can use simple error formulas for sums and products: $E(Q+R) = E(Q)+E(R)$ and
$E(QR)=(Q+E(Q))(R+E(R))-QR$. We can also simplify the computations by upper bounding a few estimates provided by Python: $E (\pi) < 4\veps_m$, $E(1/1500)<\veps_m$, $E (1.1)<\veps_m$, and $E (\sqrt{3})<2\veps_m$.
It is then a straightforward but tedious computation to check that 
$
E(|\psi(\mathbf{x})|) < 72080\veps_m,
$
which is smaller than $0.05\alpha$ for our choice of $\ell$,
and that both $E(\ell^3(|\psi(\mathbf{x})|+1.1\alpha)^p)$ and $E((2.2\alpha)^p\ell^3)$ are bounded above by $\veps_m$, whence
$$E(U_{\pm}) < (1500^3 + 17199000)\veps_m
$$
by \eqref{eq:errorRie}. We can upper bound $\veps_m$ by $2.23\cdot 10^{-16}$ (the Python commands \texttt{import sys} and then \texttt{sys.float$\_$info.epsilon}
can be used to check this), yielding a computational error bounded above by $10^{-6}$. This is much smaller than required.

\medskip

\subsection*{Acknowledgments}
\noindent
This research was conducted as part of the 2021 Fields Undergraduate Summer
Research Program. The authors are grateful to the Fields Institute for their financial support and facilitating our online collaboration. The authors would also like to thank \'Angel D. Mart\'inez and Francisco Torres de Lizaur for their guidance and suggesting this project, as well as reviewing an earlier version of this work.

\bibliographystyle{acm}
\bibliography{references.bib}

\begin{thebibliography}{1}

\bibitem{BV}
{\sc Beiner, G., and Verreault, W.}
\newblock Failure of generalized symmetry of spherical harmonics.
\newblock In preparation.

\bibitem{MR489542}
{\sc Berry, M.~V.}
\newblock Regular and irregular semiclassical wavefunctions.
\newblock {\em J. Phys. A 10}, 12 (1977), 2083--2091.

\bibitem{MR943927}
{\sc Donnelly, H., and Fefferman, C.}
\newblock Nodal sets of eigenfunctions on {R}iemannian manifolds.
\newblock {\em Invent. Math. 93}, 1 (1988), 161--183.

\bibitem{MR1900757}
{\sc Jakobson, D., and Nadirashvili, N.}
\newblock Quasi-symmetry of {$L^p$} norms of eigenfunctions.
\newblock {\em Comm. Anal. Geom. 10}, 2 (2002), 397--408.

\bibitem{jakobson2001geometric}
{\sc Jakobson, D., Nadirashvili, N., and Toth, J.}
\newblock Geometric properties of eigenfunctions.
\newblock {\em Russian Mathematical Surveys 56}, 6 (2001), 1085.

\bibitem{LogunovVideo}
{\sc Logunov, A.}
\newblock Geometry of nodal sets of {L}aplace eigenfunctions.
\newblock \url{http://scgp.stonybrook.edu/video_portal/video.php?id=4473}.

\bibitem{logunov2019review}
{\sc Logunov, A., and Malinnikova, E.}
\newblock Review of {Y}au's conjecture on zero sets of {L}aplace
  eigenfunctions.
\newblock {\em arXiv preprint arXiv:1908.01639\/} (2019).

\bibitem{eigentorus}
{\sc Mart\'{\i}nez, A.~D., and Torres~de Lizaur, F.}
\newblock Distribution symmetry of toral eigenfunctions.
\newblock {\em Rev. Mat. Iberoam. 38}, 4 (2022), 1371--1382.

\bibitem{MR1112199}
{\sc Nadirashvili, N.~S.}
\newblock Metric properties of eigenfunctions of the {L}aplace operator on
  manifolds.
\newblock {\em Ann. Inst. Fourier (Grenoble) 41}, 1 (1991), 259--265.

\end{thebibliography}
\parindent0pt
\end{document}